\documentclass[a4paper,11pt]{article}%
\usepackage{amsmath}
\usepackage{amsfonts}
\usepackage{amssymb}
\usepackage{graphicx}
\usepackage{hyperref}
\usepackage{textcomp}%
\providecommand{\U}[1]{\protect\rule{.1in}{.1in}}
\newtheorem{theorem}{Theorem}[section]

\newtheorem{corollary}[theorem]{Corollary}

\newtheorem{lemma}[theorem]{Lemma}

\newtheorem{proposition}[theorem]{Proposition}
\newtheorem{remark}[theorem]{Remark}

\numberwithin{equation}{section}
\newenvironment{proof}[1][Proof]{\noindent\textbf{#1.} }{\ \rule{0.5em}{0.5em}}
\hypersetup{
colorlinks   = true,   urlcolor     = blue,   linkcolor    = blue,   citecolor   = red }
\begin{document}

\title{The Cheeger constant as limit of Sobolev-type constants}
\author{Grey Ercole\\{\small Universidade Federal de Minas Gerais} \\{\small Belo Horizonte, MG, 30.123-970, Brazil}\\{\small grey@mat.ufmg.br}}
\maketitle

\begin{abstract}
\noindent Let $\Omega$ be a bounded, smooth domain of $\mathbb{R}^{N},$
$N\geq2.$ For $1<p<N$ and $0<q(p)<p^{\ast}:=\frac{Np}{N-p}$ let
\[
\lambda_{p,q(p)}:=\inf\left\{  \int_{\Omega}\left\vert \nabla u\right\vert
^{p}\mathrm{d}x:u\in W_{0}^{1,p}(\Omega)\text{ \ and \ }\int_{\Omega
}\left\vert u\right\vert ^{q(p)}\mathrm{d}x=1\right\}  .
\]
We prove that if $\lim_{p\rightarrow1^{+}}q(p)=1,$ then $\lim_{p\rightarrow
1^{+}}\lambda_{p,q(p)}=h(\Omega)$, where $h(\Omega)$ denotes the Cheeger
constant of $\Omega.$ Moreover, we study the behavior of the positive
solutions $w_{p,q(p)}$ to the Lane-Emden equation $-\operatorname{div}%
(\left\vert \nabla w\right\vert ^{p-2}\nabla w)=\left\vert w\right\vert
^{q-2}w,$ as $p\rightarrow1^{+}.$

\end{abstract}



\section{Introduction}

Let $\Omega$ be a smooth, bounded domain of $\mathbb{R}^{N},$ $N\geq2.$ For
$1<p<N$ and $0<q\leq p^{\ast}:=\frac{Np}{N-p}$ let
\begin{equation}
\lambda_{p,q}:=\inf\left\{  \left\Vert \nabla u\right\Vert _{p}^{p}:u\in
W_{0}^{1,p}(\Omega)\text{ \ and \ }\left\Vert u\right\Vert _{q}=1\right\}  ,
\label{lpq}%
\end{equation}
where
\[
\left\Vert u\right\Vert _{r}:=\left(  \int_{\Omega}\left\vert u\right\vert
^{r}\mathrm{d}x\right)  ^{\frac{1}{r}},\text{ \ \ }r>0.
\]
We recall that $\left\Vert \cdot\right\Vert _{r}$ is the standard norm of the
Lebesgue space $L^{r}(\Omega)$ if $r\geq1,$ but it is not a norm if $0<r<1$.

Note from (\ref{lpq}) that
\begin{equation}
\lambda_{p,q}\leq\frac{\left\Vert \nabla u\right\Vert _{p}^{p}}{\left\Vert
u\right\Vert _{q}^{p}}\text{ for all }u\in W_{0}^{1,p}(\Omega)\setminus
\left\{  0\right\}  \label{Q}%
\end{equation}
since the above quotients are homogeneous.

When $0<q<p^{\ast},$ the existence of a minimizer $u_{p,q}$ for the
constrained minimization problem (\ref{lpq}) follows from standard arguments
of the Calculus of Variations. Moreover, $u_{p,q}$ is a weak solution to the
Dirichlet problem%
\begin{equation}
\left\{
\begin{array}
[c]{lll}%
-\operatorname{div}\left(  \left\vert \nabla u\right\vert ^{p-2}\nabla
u\right)  =\lambda_{p,q}\left\vert u\right\vert ^{q-2}u & \text{\textup{in}} &
\Omega\\
u=0 & \text{\textup{on}} & \partial\Omega.
\end{array}
\right.  \label{Dir}%
\end{equation}
In consequence, $u_{p,q}>0$ in $\Omega$ and $u_{p,q}\in C^{1,\alpha}%
(\overline{\Omega})$ for some $0<\alpha<1$ (we refer to \cite[Theorem
1(i)]{GST} for the regularity of $u_{p,q}$ when $0<q<1$, in which case
(\ref{Dir}) is singular). These facts are well known when $1\leq q<p^{\ast},$
since $\lambda_{p,q}$ is the best constant in the Sobolev (compact) embedding
$W_{0}^{1,p}(\Omega)\hookrightarrow L^{q}(\Omega).$

It is worth mentioning that $u_{p,q}$ is the only positive minimizer to
(\ref{lpq}) in the sublinear case: $0<q<p.$ However, this uniqueness property
might fail in the superlinear, subcritical case: $p<q<p^{\ast}.$ For examples
and a discussion about this issue we recommend the recent paper \cite{BLind}
by Brasco and Lindgren, where an important result is established for general
smooth bounded domains: the uniqueness of the minimizer $u_{p,q}$ whenever
$2<p<q$ and $q$ is sufficiently close to $p.$ We stress that such a uniqueness
result for $1<p<2$ is not yet available in the literature.

When $q=p^{\ast},$ the infimum $\lambda_{p,p^{\ast}}$ cannot be attained in
$W_{0}^{1,p}(\Omega)$ if $\Omega\not =\mathbb{R}^{N}.$ Actually,
$\lambda_{p,p^{\ast}}$ does not depend on $\Omega$ as it coincides with the
well-known Sobolev constant $S_{N,p},$ that is:
\begin{equation}
\lambda_{p,p^{\ast}}=S_{N,p}:=N\omega_{N}^{\frac{p}{N}}\left(  \frac{N-p}%
{p-1}\right)  ^{p-1}\left(  \frac{\Gamma(N/p)\Gamma(1+N-N/p)}{\Gamma
(N)}\right)  ^{\frac{p}{N}}, \label{lpp*}%
\end{equation}
where $\Gamma(t)=\int_{0}^{\infty}s^{t-1}e^{-s}\mathrm{d}s$ is the Gamma
function and $\omega_{N}:=\pi^{N/2}/\Gamma(1+N/2)$ is the $N$-dimensional
Lebesgue volume of the unit ball of $\mathbb{R}^{N}.$

According to \cite[Theorem 9]{AFI} by Anello, Faraci and Iannizzotto, for each
fixed $p\in(1,N)$ the function
\[
(0,p^{\ast}]\ni q\mapsto\lambda_{p,q}\left\vert \Omega\right\vert ^{\frac
{p}{q}}%
\]
is decreasing and absolutely continuous on compact sets of $(0,p^{\ast}].$ The
same result, but for $q\in\lbrack1,p^{\ast}],$ had already been obtained by
Ercole in \cite{Er13}.

As for $q$ varying with $p$, Kawohl and Fridman proved in \cite{KF03} that
\begin{equation}
\lim_{p\rightarrow1^{+}}\lambda_{p,p}=h(\Omega) \label{pp}%
\end{equation}
where $h(\Omega)$ is the Cheeger constant of $\Omega.$

We recall that
\[
h(\Omega):=\inf\left\{  \frac{P(E)}{\left\vert E\right\vert }:E\subset
\overline{\Omega}\text{ and }\left\vert E\right\vert >0\right\}  ,
\]
where $P(E)$ stands for the perimeter of $E$ in $\mathbb{R}^{N}$ and
$\left\vert E\right\vert $ stands for the $N$-dimensional Lebesgue volume of
$E.$

The Cheeger problem consists of finding a subset $E\subset\overline{\Omega}$
such that
\[
h(\Omega)=\frac{P(E)}{\left\vert E\right\vert }.
\]
Such a subset $E$ is called Cheeger set of $\Omega.$

We notice from (\ref{lpp*}) that
\begin{equation}
\lim\limits_{p\rightarrow1^{+}}\lambda_{p,p^{\ast}}=N\omega_{N}^{\frac{1}{N}%
}=\left\vert \Omega\right\vert ^{\frac{1}{N}}h(\Omega^{\star}) \label{pp*}%
\end{equation}
where $\Omega^{\star}$ denotes the ball of $\mathbb{R}^{N}$ centered at the
origin and such that $\left\vert \Omega^{\star}\right\vert =\left\vert
\Omega\right\vert .$ The second equality in (\ref{pp*}) is due the fact that
balls are Cheeger sets of themselves (i.e.\thinspace they are calibrable).
Hence, as $R=(\left\vert \Omega\right\vert /\omega_{N})^{\frac{1}{N}}$ is the
radius of $\Omega^{\star},$ one has $h(\Omega^{\star})=N/R=N(\omega
_{N}/\left\vert \Omega\right\vert )^{\frac{1}{N}}.$ It is well known that
$h(\Omega^{\star})\leq h(\Omega),$ the equality occurring if and only if
$\Omega$ is a ball.

Owing to (\ref{pp}), when $p\rightarrow1^{+}$ the minimizer $u_{p,p}$
converges in $L^{1}(\Omega)$ (after passing to a subsequence) to a function
$u_{1}$ whose the $t$-superlevel sets $E_{t}:=\left\{  x\in\Omega
:u_{1}(x)>t\right\}  $ are Cheeger sets, for almost every $t>0.$ As shown in
\cite{KF03}, these properties are obtained from a variational version of the
Cheeger problem in the $BV$ setting, which we briefly present in Section
\ref{sec2}.

In this paper we suppose that $q$ varies with $p$ along a more general path,
$q=q(p)$ for $p\in(1,p^{\ast}),$ and study the behavior of $\lambda_{p,q(p)}$
when $p\rightarrow1^{+}$ and $q(p)\rightarrow1.$ Adapting an estimate from
\cite{Er13} (see Lemma \ref{estim} below) and making use of (\ref{pp}) we
extend the results of \cite{KF03}. Our main result, which will be proved in
Section \ref{sec3}, is stated as follows.

\begin{theorem}
\label{main}If $0<q(p)<p^{\ast}$ and $\lim\limits_{p\rightarrow1^{+}}q(p)=1,$
then
\begin{equation}
\lim_{p\rightarrow1^{+}}\lambda_{p,q(p)}=h(\Omega), \label{l=h}%
\end{equation}%
\begin{equation}
\lim_{p\rightarrow1^{+}}\left\Vert u_{p,q(p)}\right\Vert _{1}=1 \label{l1=1}%
\end{equation}
and%
\begin{equation}
\lim_{p\rightarrow1^{+}}\left\Vert u_{p,q(p)}\right\Vert _{\infty}^{q(p)-p}=1.
\label{linf2}%
\end{equation}

Moreover, any sequence $\left(  u_{p_{n},q(p_{n})}\right)  ,$ with
$p_{n}\rightarrow1^{+},$ admits a subsequence that converges in $L^{1}%
(\Omega)$ to a nonnegative function $u\in L^{1}(\Omega)\cap L^{\infty}%
(\Omega)$ such that:

\begin{enumerate}
\item[(a)] $\left\Vert u\right\Vert _{1}=1,$

\item[(b)] $\dfrac{1}{\left\vert \Omega\right\vert }\leq\left\Vert
u\right\Vert _{\infty}\leq\dfrac{h(\Omega)^{N}}{\left\vert \Omega\right\vert
h(\Omega^{\star})^{N}},$ and

\item[(c)] for almost every $t\geq0,$ the $t$-superlevel set
\[
E_{t}:=\left\{  x\in\Omega:u(x)>t\right\}
\]
is a Cheeger set.
\end{enumerate}
\end{theorem}

Besides allowing $q(p)\rightarrow1^{-}$, which embraces (\ref{Dir}) in its
singular form, our approach holds for every family of extremals $u_{p,q(p)}$
in the superlinear, subcritical case: $p<q(p)<p^{\ast}.$

It is simple to verify that the function
\begin{equation}
v_{p,q}:=\lambda_{p,q}^{\frac{1}{q-p}}u_{p,q}, \label{vpq}%
\end{equation}
is a positive weak solution to the Lane-Emden type problem%
\begin{equation}
\left\{
\begin{array}
[c]{lll}%
-\operatorname{div}\left(  \left\vert \nabla w\right\vert ^{p-2}\nabla
w\right)  =\left\vert w\right\vert ^{q-2}w & \text{\textup{in}} & \Omega\\
w=0 & \text{\textup{on}} & \partial\Omega.
\end{array}
\right.  \label{Dir1}%
\end{equation}

The next corollary is stated to solutions to (\ref{Dir1}) in the form
(\ref{vpq}). It is an immediate consequence of (\ref{l=h}) and (\ref{linf2})
since
\[
\left\Vert v_{p,q}\right\Vert _{q}^{q-p}=\lambda_{p,q}\left\Vert
u_{p,q}\right\Vert _{q}^{q-p}=\lambda_{p,q}\text{ \ and \ }\left\Vert
v_{p,q}\right\Vert _{\infty}^{q-p}=\lambda_{p,q}\left\Vert u_{p,q}\right\Vert
_{\infty}^{q-p}.
\]

\begin{corollary}
\label{LE1}If $0<q(p)<p^{\ast}$ and $\lim\limits_{p\rightarrow1^{+}}q(p)=1,$
then%
\[
\lim_{p\rightarrow1^{+}}\left\Vert v_{p,q(p)}\right\Vert _{q(p)}%
^{q(p)-p}=h(\Omega)=\lim_{p\rightarrow1^{+}}\left\Vert v_{p,q(p)}\right\Vert
_{\infty}^{q(p)-p}.
\]

\end{corollary}

In the particular case where $q(p)\equiv1$ this result had already been
obtained in \cite{BE11} by Bueno and Ercole, without using (\ref{pp}).

As it is well known, (\ref{Dir1}) has a unique positive weak solution when
$0<q<p,$ which is, of course, that given by (\ref{vpq}). However, there are
examples of smooth domains for which (\ref{Dir1}) has multiple positive weak
solutions when $p<q<p^{\ast},$ which may be of the form (\ref{vpq}) or not
(see \cite{BLind} and references therein). By the way, it is plain to check
that%
\begin{equation}
\lambda_{p,q}=\left\Vert v_{p,q}\right\Vert _{q}^{q-p}=\min\left\{  \left\Vert
w\right\Vert _{q}^{q-p}:w\text{ is a weak solution to }(\ref{Dir1})\right\}  .
\label{mLE}%
\end{equation}

Corollary \ref{LE1} deals with the behavior of positive weak solutions to
(\ref{Dir1}) that attains the minimum in (\ref{mLE}). Aiming to cover a wider
class of positive weak solutions $w_{p,q}$ to (\ref{Dir1}), including those
satisfying $\left\Vert w_{p,q}\right\Vert _{q}^{q-p}>\lambda_{p,q}$, we
provide the following stronger result, which will be proved in Section
\ref{sec4} by using Picone's inequality (see \cite{AH,BF}).

\begin{theorem}
\label{main2}Let $w_{p,q(p)}\in W_{0}^{1,p}(\Omega)$ be a positive weak
solution to (\ref{Dir1}), with $p<q(p)<p^{\ast}.$ Then, either
\begin{equation}
\limsup_{p\rightarrow1^{+}}\left\Vert w_{p,q(p)}\right\Vert _{\infty}%
^{q(p)-p}=+\infty\label{alt}%
\end{equation}
or%
\[
\lim_{p\rightarrow1^{+}}\left\Vert w_{p,q(p)}\right\Vert _{q(p)}%
^{q(p)-p}=h(\Omega)=\lim_{p\rightarrow1^{+}}\left\Vert w_{p,q(p)}\right\Vert
_{\infty}^{q(p)-p}.
\]

\end{theorem}

The alternative (\ref{alt}) can be replaced with (see Remark \ref{alter})
\[
\limsup_{p\rightarrow1^{+}}\left\Vert w_{p,q(p)}\right\Vert _{q(p)}%
^{q(p)-p}=+\infty.
\]
We believe that determining whether this alternative (or its equivalent
version (\ref{alt})) is actually possible is a very interesting open question
that we plan to study in the near future.

\section{The Cheeger problem in the $BV$ setting\label{sec2}}

In this section, we assume that $\Omega$ is a Lipschitz bounded domain and
collect some definitions, properties and basic results related to the
variational version of the Cheeger problem in the $BV$ setting. For details we
refer to \cite{CC} and \cite{Pa11}.

The total variation of $u\in L^{1}(\Omega)$ is defined as
\[
\left\vert Du\right\vert (\Omega):=\sup\left\{  \int_{\Omega}%
u\operatorname{div}\varphi\mathrm{d}x:\varphi\in C_{c}^{1}(\Omega
;\mathbb{R}^{N})\text{ \ and \ }\left\Vert \varphi\right\Vert _{L^{\infty}%
}\leq1\right\}  .
\]

The space $BV(\Omega)$ of the functions $u\in L^{1}(\Omega)$ of bounded
variation in $\Omega$ (i.e.$\,\left\vert Du\right\vert (\Omega)<\infty$),
endowed with the norm
\[
\left\Vert u\right\Vert _{BV}:=\left\Vert u\right\Vert _{1}+\left\vert
Du\right\vert (\Omega),
\]
is a Banach space compactly embedded into $L^{1}(\Omega)$. Moreover, the
functional $BV(\Omega)\ni u\mapsto\left\vert Du\right\vert (\Omega)$ is lower
semicontinuous in $L^{1}(\Omega).$

The Cheeger constant is also characterized as (see \cite[Proposition
3.1]{Pa11})
\begin{equation}
h(\Omega)=\inf_{BV_{0}(\Omega)}\frac{\left\vert Du\right\vert (\mathbb{R}%
^{N})}{\left\Vert u\right\Vert _{1}} \label{hbv}%
\end{equation}
where
\[
BV_{0}(\Omega):=\left\{  u\in BV(\mathbb{R}^{N}):\left\Vert u\right\Vert
_{1}>0\text{ \ and \ }u\equiv0\text{ in }\mathbb{R}^{N}\setminus
\overline{\Omega}\right\}
\]
and%
\[
\left\vert Du\right\vert (\mathbb{R}^{N})=\left\vert Du\right\vert
(\Omega)+\int_{\partial\Omega}\left\vert u\right\vert \mathrm{d}%
\mathcal{H}^{N-1}%
\]
($\mathcal{H}^{N-1}$ stands for the $(N-1)$-Hausdorff measure in
$\mathbb{R}^{N}$).

\begin{proposition}
[{\cite[Corollary 1(2)]{CC}}]\label{bv} Let $\left(  u_{n}\right)  \subset
BV_{0}(\Omega)$ be such that $u_{n}\rightarrow u$ in $L^{1}(\mathbb{R}^{N}).$
Then
\[
\left\vert Du\right\vert (\mathbb{R}^{N})\leq\liminf_{n\rightarrow\infty
}\left\vert Du_{n}\right\vert (\mathbb{R}^{N}).
\]

\end{proposition}

\begin{proposition}
\label{bv1}Suppose that
\[
h(\Omega)=\frac{\left\vert Du\right\vert (\mathbb{R}^{N})}{\left\Vert
u\right\Vert _{1}}%
\]
for some $u\in BV_{0}(\Omega).$ Then
\[
E_{t}:=\left\{  x\in\Omega:u(x)>t\right\}
\]
is a Cheeger set for almost every $t\geq0.$

Inversely, if $E\subset\overline{\Omega}$ is a Cheeger set of $\Omega,$ then
\[
h(\Omega)=\frac{\left\vert D\chi_{E}\right\vert (\mathbb{R}^{N})}{\left\Vert
\chi_{E}\right\Vert _{1}}%
\]
where $\chi_{E}$ stands for the characteristic function of $E$ in
$\mathbb{R}^{N}.$
\end{proposition}

\begin{proof}
Combining Coarea formula and Cavalieri's principle we find
\begin{equation}
0=\left\vert Du\right\vert (\mathbb{R}^{N})-h(\Omega)\left\Vert u\right\Vert
_{1}=\int_{0}^{\infty}(P(E_{t})-h(\Omega)\left\vert E_{t}\right\vert
)\mathrm{d}t. \label{prelima}%
\end{equation}

As $\left\vert E_{t}\right\vert >0$ a.e.$\,t\geq0$ we have that $P(E_{t}%
)-h(\Omega)\left\vert E_{t}\right\vert \geq0$ a.e.$\,t\geq0.$ Therefore, it
follows from (\ref{prelima}) that
\[
h(\Omega)=\frac{P(E_{t})}{\left\vert E_{t}\right\vert }\text{ \ a.e.\thinspace
}t\geq0.
\]

Now, if $E\subset\overline{\Omega}$ is a Cheeger set of $\Omega,$ then
$\chi_{E}\in BV_{0}(\Omega).$ As $P(E)=\left\vert D\chi_{E}\right\vert
(\mathbb{R}^{N})$ and $\left\Vert \chi_{E}\right\Vert _{1}=\left\vert
E\right\vert $ we have
\[
h(\Omega)=\frac{P(E)}{\left\vert E\right\vert }=\frac{\left\vert D\chi
_{E}\right\vert (\mathbb{R}^{N})}{\left\Vert \chi_{E}\right\Vert _{1}}.
\]

\end{proof}

\section{Proof of Theorem \ref{main}\label{sec3}}

We recall from the Introduction that $u_{p,q}$ (for $1<p<N$ and $0<q<p^{\ast}%
$) denotes the positive minimizer of the constrained minimization problem
(\ref{lpq}), so that $u_{p,q}\in W_{0}^{1,p}(\Omega),$
\[
u_{p,q}>0\text{ in }\Omega,\text{ \ }\left\Vert u_{p,q}\right\Vert
_{q}=1,\text{ \ }\lambda_{p,q}=\left\Vert \nabla u_{p,q}\right\Vert _{p}^{p},
\]
and $u_{p,q}$ is a weak solution to (\ref{Dir}).

If $q=p$ the Dirichlet problem (\ref{Dir}) is homogeneous and thus it\ can be
recognized as an eigenvalue problem. In this setting, $\lambda_{p,p}$ is known
as the first eigenvalue of the Dirichlet $p$-Laplacian. Actually,
$\lambda_{p,p}$ is simple no sense that the set of its corresponding
eigenfunctions is generated by $u_{p,p},$ that is: $w\in W_{0}^{1,p}(\Omega)$
is a nontrivial weak solution to%
\begin{equation}
\left\{
\begin{array}
[c]{lll}%
-\operatorname{div}\left(  \left\vert \nabla u\right\vert ^{p-2}\nabla
u\right)  =\lambda_{p,p}\left\vert u\right\vert ^{p-2}u & \text{\textup{in}} &
\Omega\\
u=0 & \text{\textup{on}} & \partial\Omega
\end{array}
\right.  \label{Dirp}%
\end{equation}
if and only if $w=ku_{p,p}$ for some $k\in\mathbb{R}\setminus\left\{
0\right\}  .$

In this section, we prove Theorem \ref{main} by assuming that $\partial\Omega$
is smooth enough to ensure that $u_{p,q}\in C^{1}(\overline{\Omega}).$ In
consequence, $u_{p,q}\in BV_{0}(\Omega)$ (after extended as zero on
$\mathbb{R}^{N}\setminus\overline{\Omega}$) and%
\[
\left\vert Du_{p,q}\right\vert (\mathbb{R}^{N})=\left\Vert \nabla
u_{p,q}\right\Vert _{1},
\]
since%
\[
\left\vert Du_{p,q}\right\vert (\Omega)=\left\Vert \nabla u_{p,q}\right\Vert
_{1}\text{ \ and \ }\int_{\partial\Omega}\left\vert u_{p,q}\right\vert
\mathrm{d}\mathcal{H}^{N-1}=0.
\]

The next result is adapted from Lemma 5 of \cite{Er13} established there for
$1\leq q<p^{\ast}$.

\begin{lemma}
\label{estim}Let $u\in W_{0}^{1,p}(\Omega)\cap C^{1}(\overline{\Omega})$ be a
positive weak solution to the Dirichlet problem
\begin{equation}
\left\{
\begin{array}
[c]{lll}%
-\operatorname{div}\left(  \left\vert \nabla u\right\vert ^{p-2}\nabla
u\right)  =\lambda\left\vert u\right\vert ^{q-2}u & \text{\textup{in}} &
\Omega\\
u=0 & \text{\textup{on}} & \partial\Omega
\end{array}
\right.  \label{Dirl}%
\end{equation}
with $0\leq q<p^{\ast}$ and $\lambda>0.$ If $\sigma\geq1,$ then%
\begin{equation}
C_{\lambda,\sigma,q}\left\Vert u\right\Vert _{\infty}^{\frac{N(p-q)+p\sigma
}{p}}\leq\left\Vert u\right\Vert _{\sigma}^{\sigma}, \label{est9}%
\end{equation}
where%
\[
C_{\lambda,\sigma,q}:=\left(  \frac{\lambda_{p,p^{\ast}}}{\lambda}\right)
^{\frac{N}{p}}\left(  \frac{p}{p+N(p-1)}\right)  ^{N+1}I_{\sigma,q}%
\]
and%
\[
I_{\sigma,q}:=\left\{
\begin{array}
[c]{lll}%
\sigma\int_{0}^{1}\left(  1-\tau\right)  ^{\frac{N(p-1)}{p}}\tau
^{\frac{N(1-q)}{p}+\sigma-1}\mathrm{d}\tau & \text{if} & 0\leq q<1\\
\sigma%
{\displaystyle\int_{0}^{1}}
(1-\tau)^{\frac{N(p-1)}{p}}\tau^{\sigma-1}\mathrm{d}\tau & \text{if} & 1\leq
q<p^{\ast}.
\end{array}
\right.
\]

\end{lemma}

\begin{proof}
For each $0<t<\left\Vert u\right\Vert _{\infty}$ let us define
\[
(u-t)_{+}:=\max\left\{  u-t,0\right\}  \text{ \ and \ }A_{t}:=\left\{
x\in\Omega:u(x)>t\right\}  .
\]

As $(u-t)_{+}\in W_{0}^{1,p}(\Omega)$ and $u$ is a positive weak solution to
(\ref{Dirl}), we have%
\begin{equation}
\int_{A_{t}}\left\vert \nabla u\right\vert ^{p}\mathrm{d}x=\lambda\int
_{\Omega}\left\vert \nabla u\right\vert ^{p-2}\nabla u\cdot\nabla
(u-t)_{+}\mathrm{d}x=\lambda\int_{\Omega}u^{q-1}(u-t)_{+}\mathrm{d}%
x=\lambda\int_{A_{t}}u^{q-1}(u-t)\mathrm{d}x. \label{est00}%
\end{equation}

We also have%
\begin{equation}
\left(  \int_{A_{t}}(u-t)\mathrm{d}x\right)  ^{p}\leq\left\vert A_{t}%
\right\vert ^{p-\frac{p}{p^{\ast}}}\left(  \int_{A_{t}}(u-t)^{p^{\ast}%
}\mathrm{d}x\right)  ^{\frac{p}{p^{\ast}}}\leq\frac{\left\vert A_{t}%
\right\vert ^{p-\frac{p}{p^{\ast}}}}{\lambda_{p,p^{\ast}}}\int_{A_{t}%
}\left\vert \nabla u\right\vert ^{p}\mathrm{d}x, \label{est2}%
\end{equation}
where we have used H\"{o}lder's inequality and (\ref{Q}). Note that
\[
\lambda_{p,p^{\ast}}\leq\frac{\left\Vert \nabla(u-t)_{+}\right\Vert _{p}^{p}%
}{\left\Vert (u-t)_{+}\right\Vert _{p^{\ast}}^{p}}=\frac{%
{\displaystyle\int_{A_{t}}}
\left\vert \nabla u\right\vert ^{p}\mathrm{d}x}{\left(
{\displaystyle\int_{A_{t}}}
(u-t)^{p^{\ast}}\mathrm{d}x\right)  ^{\frac{p}{p^{\ast}}}}.
\]

We divide the remaining of the proof in two cases.

\textbf{Case 1:} $0\leq q<1.$ As
\begin{equation}
\int_{A_{t}}u^{q-1}(u-t)\mathrm{d}x\leq t^{q-1}\int_{A_{t}}(\left\Vert
u\right\Vert _{\infty}-t)\mathrm{d}x \label{est0}%
\end{equation}
we obtain from (\ref{est00}) the estimate%
\begin{equation}
\int_{A_{t}}\left\vert \nabla u\right\vert ^{p}\mathrm{d}x\leq\lambda
t^{q-1}(\left\Vert u\right\Vert _{\infty}-t)\left\vert A_{t}\right\vert .
\label{est1}%
\end{equation}

Combining (\ref{est1}) and (\ref{est2}), we obtain the inequalities%
\[
\lambda_{p,p^{\ast}}\left\vert A_{t}\right\vert ^{-p+\frac{p}{p^{\ast}}%
}\left(  \int_{A_{t}}(u-t)\mathrm{d}x\right)  ^{p}\leq\int_{A_{t}}\left\vert
\nabla u\right\vert ^{p}\mathrm{d}x\leq\lambda t^{q-1}(\left\Vert u\right\Vert
_{\infty}-t)\left\vert A_{t}\right\vert
\]
which lead to
\begin{equation}
\frac{\lambda_{p,p^{\ast}}t^{1-q}}{\lambda(\left\Vert u\right\Vert _{\infty
}-t)}\left(  \int_{A_{t}}(u-t)\mathrm{d}x\right)  ^{p}\leq\left\vert
A_{t}\right\vert ^{p(1-\frac{1}{p^{\ast}}+\frac{1}{p})}=\left\vert
A_{t}\right\vert ^{p(\frac{N-1}{N})}. \label{est3}%
\end{equation}

Now, let us define the function%
\[
g(t):=\int_{A_{t}}(u-t)\mathrm{d}x.
\]

It is simple to verify that%
\[
g(t)=\int_{t}^{\left\Vert u\right\Vert _{\infty}}\left\vert A_{s}\right\vert
\mathrm{d}s,
\]
so that%
\[
g^{\prime}(t)=-\left\vert A_{t}\right\vert .
\]

Then, (\ref{est3}) can be rewritten as%
\begin{equation}
\left(  \frac{\lambda_{p,p^{\ast}}}{\lambda}\right)  ^{\frac{N}{p(N+1)}%
}\left(  \frac{t^{1-q}}{\left\Vert u\right\Vert _{\infty}-t}\right)
^{\frac{N}{p(N+1)}}\leq-g^{\prime}(t)g(t)^{-\frac{N}{N+1}}. \label{est4}%
\end{equation}

Integration of the right-hand side of (\ref{est4}) over $[t,\left\Vert
u\right\Vert _{\infty}]$ yields
\begin{equation}
-\int_{t}^{\left\Vert u\right\Vert _{\infty}}g^{\prime}(s)g(s)^{-\frac{N}%
{N-1}}\mathrm{d}s=(N+1)g(t)^{\frac{1}{N+1}}-(N+1)g(\left\Vert u\right\Vert
_{\infty})^{\frac{1}{N+1}}\leq(N+1)g(t)^{\frac{1}{N+1}} \label{est5}%
\end{equation}
whereas integration of the function at the left-hand side of (\ref{est4}) over
$[t,\left\Vert u\right\Vert _{\infty}]$ yields
\begin{align}
\int_{t}^{\left\Vert u\right\Vert _{\infty}}\left(  \frac{s^{1-q}}{\left\Vert
u\right\Vert _{\infty}-s}\right)  ^{\frac{N}{p(N+1)}}\mathrm{d}s  &  \geq
t^{\frac{N(1-q)}{p(N+1)}}\int_{t}^{\left\Vert u\right\Vert _{\infty}}\left(
\left\Vert u\right\Vert _{\infty}-s\right)  ^{-\frac{N}{p(N+1)}}%
\mathrm{d}s\label{est6}\\
&  =\frac{p(N+1)}{p+N(p-1)}t^{\frac{N(1-q)}{p(N+1)}}\left(  \left\Vert
u\right\Vert _{\infty}-t\right)  ^{\frac{p+N(p-1)}{p(N+1)}}.\nonumber
\end{align}

Thus, after integrating (\ref{est4}) we obtain from (\ref{est5}) and
(\ref{est6}) the inequality%
\begin{equation}
\left(  \frac{\lambda_{p,p^{\ast}}}{\lambda}\right)  ^{\frac{N}{p}}\left(
\frac{p}{p+N(p-1)}\right)  ^{N+1}t^{\frac{N(1-q)}{p}}\left(  \left\Vert
u\right\Vert _{\infty}-t\right)  ^{\frac{p+N(p-1)}{p}.}\leq g(t). \label{est7}%
\end{equation}

As $g(t)\leq(\left\Vert u\right\Vert _{\infty}-t)\left\vert A_{t}\right\vert
$, it follows from (\ref{est7}) that
\[
\left(  \frac{\lambda_{p,p^{\ast}}}{\lambda}\right)  ^{\frac{N}{p}}\left(
\frac{p}{p+N(p-1)}\right)  ^{N+1}t^{\frac{N(1-q)}{p}}\left(  \left\Vert
u\right\Vert _{\infty}-t\right)  ^{\frac{N(p-1)}{p}}\leq\left\vert
A_{t}\right\vert .
\]

Now, for a given $\sigma\geq1,$ we multiply the latter inequality by $\sigma
t^{\sigma-1}$ and integrate over $[0,\left\Vert u\right\Vert _{\infty}]$ to
get (\ref{est9}) after noticing that%
\[
\sigma\int_{0}^{\left\Vert u\right\Vert _{\infty}}\left\vert A_{t}\right\vert
t^{\sigma-1}\mathrm{d}t=\int_{\Omega}u^{\sigma}\mathrm{d}x,
\]
and that the change of variable $t=\left\Vert u\right\Vert _{\infty}\tau$
yields%
\[
\int_{0}^{\left\Vert u\right\Vert _{\infty}}t^{\frac{N(1-q)}{p}+\sigma
-1}\left(  \left\Vert u\right\Vert _{\infty}-t\right)  ^{\frac{N(p-1)}{p}%
}\mathrm{d}t=\left\Vert u\right\Vert _{\infty}^{\frac{N(p-q)}{p}+\sigma}%
\int_{0}^{1}\left(  1-\tau\right)  ^{\frac{N(p-1)}{p}}\tau^{\frac{N(1-q)}%
{p}+\sigma-1}\mathrm{d}\tau.
\]

\textbf{Case 2:} $1\leq q<p^{\ast}.$ The factor $t^{q-1}$ in (\ref{est0}) can
be replaced with $\left\Vert u\right\Vert _{\infty}^{q-1},$ so that
(\ref{est4}) and (\ref{est6}) become
\[
\left(  \frac{\lambda_{p,p^{\ast}}}{\lambda}\right)  ^{\frac{N}{p(N+1)}%
}\left(  \frac{\left\Vert u\right\Vert _{\infty}^{1-q}}{\left\Vert
u\right\Vert _{\infty}-t}\right)  ^{\frac{N}{p(N+1)}}\leq-g^{\prime
}(t)g(t)^{-\frac{N}{N+1}}%
\]
and
\[
\int_{t}^{\left\Vert u\right\Vert _{\infty}}\left(  \frac{\left\Vert
u\right\Vert _{\infty}^{1-q}}{\left\Vert u\right\Vert _{\infty}-s}\right)
^{\frac{N}{p(N+1)}}\mathrm{d}s=\frac{p(N+1)}{p+N(p-1)}\left\Vert u\right\Vert
_{\infty}^{\frac{N(1-q)}{p(N+1)}}\left(  \left\Vert u\right\Vert _{\infty
}-t\right)  ^{\frac{p+N(p-1)}{p(N+1)}},
\]
respectively. Hence, we obtain from (\ref{est5}) that%
\[
\left(  \frac{\lambda_{p,p^{\ast}}}{\lambda}\right)  ^{\frac{N}{p(N+1)}}%
\frac{p(N+1)}{p+N(p-1)}\left\Vert u\right\Vert _{\infty}^{\frac{N(1-q)}%
{p(N+1)}}\left(  \left\Vert u\right\Vert _{\infty}-t\right)  ^{\frac
{p+N(p-1)}{p(N+1)}}\leq(N+1)g(t)^{\frac{1}{N+1}}.
\]

Then, using that $g(t)\leq(\left\Vert u\right\Vert _{\infty}-t)\left\vert
A_{t}\right\vert $ the latter inequality leads to
\begin{equation}
\left(  \frac{\lambda_{p,p^{\ast}}}{\lambda}\right)  ^{\frac{N}{p}}\left(
\frac{p}{p+N(p-1)}\right)  ^{N+1}\left\Vert u\right\Vert _{\infty}%
^{\frac{N(1-q)}{p}}\left(  \left\Vert u\right\Vert _{\infty}-t\right)
^{\frac{N(p-1)}{p}}\leq\left\vert A_{t}\right\vert . \label{est8a}%
\end{equation}

Multiplying (\ref{est8a}) by $\sigma t^{\sigma-1}$ and integrate over
$[0,\left\Vert u\right\Vert _{\infty}]$ we arrive at (\ref{est9}) with
\[
I_{\sigma,q}=\sigma%
{\displaystyle\int_{0}^{1}}
(1-\tau)^{\frac{N(p-1)}{p}}\tau^{\sigma-1}\mathrm{d}\tau.
\]

\end{proof}

\begin{remark}
\label{s=q}The estimate (\ref{est9}) can be rewritten as%
\[
C_{\lambda,\sigma,q}\left\Vert u\right\Vert _{\infty}^{\frac{N}{p^{\ast}%
}(p^{\ast}-q)+(\sigma-q)}\leq\left\Vert u\right\Vert _{\sigma}^{\sigma}.
\]

\end{remark}

In the sequel, $e_{p}$ denotes the $L^{\infty}$-normalized minimizer
corresponding to $\lambda_{p,p}$, that is:%
\begin{equation}
e_{p}:=\dfrac{u_{p,p}}{\left\Vert u_{p,p}\right\Vert _{\infty}}. \label{ep}%
\end{equation}

As $e_{p}$ is also a positive weak solution to the homogeneous Dirichlet
problem (\ref{Dirp}), Lemma \ref{estim} applied to $e_{p,}$ with $q=p,$
$\sigma=1$ and $\lambda=\lambda_{p,p},$ yields%
\[
\left(  \frac{\lambda_{p,p^{\ast}}}{\lambda_{p,p}}\right)  ^{\frac{N}{p}%
}\left(  \frac{p}{p+N(p-1)}\right)  ^{N+1}%
{\displaystyle\int_{0}^{1}}
(1-\tau)^{\frac{N(p-1)}{p}}\mathrm{d}\tau\leq\left\Vert e_{p}\right\Vert
_{1}.
\]

Hence, we have%
\begin{equation}
0<\left\vert \Omega\right\vert \left(  \frac{h(\Omega^{\star})}{h(\Omega
)}\right)  ^{N}\leq\liminf_{p\rightarrow1^{+}}\left\Vert e_{p}\right\Vert
_{1}, \label{blamb1}%
\end{equation}
since%
\[
\lim_{p\rightarrow1^{+}}\left(  \frac{\lambda_{p,p^{\ast}}}{\lambda_{p,p}%
}\right)  ^{\frac{N}{p}}=\left\vert \Omega\right\vert \left(  \frac
{h(\Omega^{\star})}{h(\Omega)}\right)  ^{N}%
\]
and
\begin{equation}
\lim_{p\rightarrow1^{+}}\left(  \frac{p}{p+N(p-1)}\right)  ^{N+1}%
=\lim_{p\rightarrow1^{+}}%
{\displaystyle\int_{0}^{1}}
(1-\tau)^{\frac{N(p-1)}{p}}\mathrm{d}\tau=1. \label{blamb5}%
\end{equation}

\begin{lemma}
\label{lem1}If $q_{n}\rightarrow1$ and $p_{n}\rightarrow1^{+},$ then (up to a
subsequence) $e_{p_{n}}$ converges in $L^{1}(\Omega)$ to a function $e.$
Moreover,
\begin{equation}
\lim_{n\rightarrow\infty}\int_{\Omega}e_{p_{n}}^{q_{n}}\mathrm{d}%
x=\lim_{n\rightarrow\infty}\int_{\Omega}e_{p_{n}}^{p_{n}}\mathrm{d}%
x=\left\Vert e\right\Vert _{1}>0. \label{blamb4}%
\end{equation}

\end{lemma}

\begin{proof}
We have $\left\Vert e_{p}\right\Vert _{1}\leq\left\Vert e_{p}\right\Vert
_{\infty}\left\vert \Omega\right\vert =\left\vert \Omega\right\vert $ and, by
H\"{o}lder inequality,
\[
\left\Vert \nabla e_{p}\right\Vert _{1}\leq\left\Vert \nabla e_{p}\right\Vert
_{p}\left\vert \Omega\right\vert ^{1-\frac{1}{p}}=\lambda_{p,p}^{\frac{1}{p}%
}\left\vert \Omega\right\vert ^{1-\frac{1}{p}}.
\]
Hence, it follows from (\ref{pp}) that the family $\left(  e_{p}\right)  $ is
uniformly bounded in $BV(\Omega).$ Therefore, owing to the compactness of the
embedding $BV(\Omega)\hookrightarrow L^{1}(\Omega)$, we can assume that (up to
a subsequence) $e_{p_{n}}$ converges to a function $e$ in $L^{1}(\Omega)$ and
also poinwise almost everywhere in $\Omega.$ In view of (\ref{blamb1}) the
convergence in $L^{1}(\Omega)$ shows that $\left\Vert e\right\Vert _{1}>0.$ As
the nonnegative functions $e_{p_{n}}^{q_{n}}$ and $e_{p_{n}}^{p_{n}}$ are
dominated by $1,$ the convergence a.e.\thinspace in $\Omega$ leads to the
equalities in (\ref{blamb4}).
\end{proof}

\begin{lemma}
\label{Pa}If $0<q(p)<p^{\ast}$ and $\lim\limits_{p\rightarrow1^{+}}q(p)=1,$
then%
\begin{equation}
\limsup_{p\rightarrow1^{+}}\lambda_{p,q(p)}\leq h(\Omega) \label{l<h}%
\end{equation}
and
\begin{equation}
\limsup_{p\rightarrow1^{+}}\left\Vert u_{p,q(p)}\right\Vert _{1}\leq1.
\label{l1<1}%
\end{equation}

\end{lemma}

\begin{proof}
Let us take $p_{n}\rightarrow1^{+}$ such that
\[
\lim_{n\rightarrow\infty}\lambda_{p_{n},q(p_{n})}=L:=\limsup_{p\rightarrow
1^{+}}\lambda_{p,q(p)}.
\]

Using (\ref{Q}) for $\lambda_{p_{n},q(p_{n})}$ and the definition of
$e_{p_{n}}$ we have that
\[
\lambda_{p_{n},q(p_{n})}\leq\frac{\left\Vert \nabla e_{p_{n}}\right\Vert
_{p_{n}}^{p_{n}}}{\left\Vert e_{p_{n}}\right\Vert _{q(p_{n})}^{p_{n}}}%
=\lambda_{p_{n},p_{n}}\left(  \frac{\left\Vert e_{p_{n}}\right\Vert _{p_{n}}%
}{\left\Vert e_{p_{n}}\right\Vert _{q(p_{n})}}\right)  ^{p_{n}}.
\]
Hence, we can apply Lemma \ref{lem1} to get (\ref{l<h}) from (\ref{pp}),
since
\[
L=\lim_{n\rightarrow\infty}\lambda_{p_{n},q(p_{n})}\leq\lim_{n\rightarrow
\infty}\lambda_{p_{n},p_{n}}\frac{\lim_{n\rightarrow\infty}\left\Vert
e_{p_{n}}\right\Vert _{p_{n}}}{\lim_{n\rightarrow\infty}\left\Vert e_{p_{n}%
}\right\Vert _{q(p_{n})}}=\lim_{n\rightarrow\infty}\lambda_{p_{n},p_{n}%
}=h(\Omega).
\]

Using H\"{o}lder's inequality and exploiting (\ref{Q}) with respect to
$\lambda_{p,p},$ we obtain%
\begin{equation}
\left\Vert u_{p,q(p)}\right\Vert _{1}\leq\left\Vert u_{p,q(p)}\right\Vert
_{p}\left\vert \Omega\right\vert ^{1-\frac{1}{p}}\leq\lambda_{p,p}^{-\frac
{1}{p}}\left\Vert \nabla u_{p,q(p)}\right\Vert _{p}\left\vert \Omega
\right\vert ^{1-\frac{1}{p}}=\lambda_{p,p}^{-\frac{1}{p}}\lambda
_{p,q(p)}^{\frac{1}{p}}\left\vert \Omega\right\vert ^{1-\frac{1}{p}}.
\label{pb}%
\end{equation}

Hence, (\ref{l1<1}) follows from (\ref{pp}) and (\ref{l<h}).
\end{proof}

\begin{lemma}
\label{linf}If $0<q(p)<p^{\ast}$ and $\lim\limits_{p\rightarrow1^{+}}q(p)=1,$
then%
\begin{equation}
\frac{1}{\left\vert \Omega\right\vert }\leq\liminf_{p\rightarrow1^{+}%
}\left\Vert u_{p,q(p)}\right\Vert _{\infty}\text{ \ and \ }\limsup
_{p\rightarrow1^{+}}\left\Vert u_{p,q(p)}\right\Vert _{\infty}\leq
\frac{h(\Omega)^{N}}{\left\vert \Omega\right\vert h(\Omega^{\star})^{N}}.
\label{linf1}%
\end{equation}

\end{lemma}

\begin{proof}
The first estimate in (\ref{linf1}) is immediate since
\[
1=\left\Vert u_{p,q(p)}\right\Vert _{q(p)}^{q(p)}\leq\left\Vert u_{p,q(p)}%
\right\Vert _{\infty}^{q(p)}\left\vert \Omega\right\vert .
\]

According to Remark \ref{s=q} we have that%
\begin{equation}
C_{p}\left\Vert u_{p,q(p)}\right\Vert _{\infty}^{\frac{N(p^{\ast}%
-q(p))}{p^{\ast}}}\leq\left\Vert u_{p,q(p)}\right\Vert _{q(p)}^{q(p)}=1
\label{lf1}%
\end{equation}
where%
\[
C_{p}:=\left(  \frac{\lambda_{p,p^{\ast}}}{\lambda_{p,q(p)}}\right)
^{\frac{N}{p}}\left(  \frac{p}{p+N(p-1)}\right)  ^{N+1}I_{p}%
\]
and
\[
I_{p}:=\left\{
\begin{array}
[c]{lll}%
q(p)\int_{0}^{1}\left(  1-\tau\right)  ^{\frac{N(p-1)}{p}}\tau^{\frac
{N(1-q(p))}{p}+q(p)-1}\mathrm{d}\tau & \text{if} & 0\leq q(p)<1\\
q(p)%
{\displaystyle\int_{0}^{1}}
(1-\tau)^{\frac{N(p-1)}{p}}\tau^{q(p)-1}\mathrm{d}\tau & \text{if} & 1\leq
q(p)<p^{\ast}.
\end{array}
\right.
\]

It follows from (\ref{lf1}) that%
\[
\left\Vert u_{p,q(p)}\right\Vert _{\infty}\leq C_{p}^{-\frac{p^{\ast}%
}{N(p^{\ast}-q(p))}}.
\]
As%
\[
\lim_{p\rightarrow1^{+}}\frac{p^{\ast}}{N(p^{\ast}-q(p))}=1=\lim
_{p\rightarrow1^{+}}\left(  \frac{p+N(p-1)}{p}\right)  ^{N+1}=\lim
_{p\rightarrow1^{+}}I_{p},
\]
and%
\[
C_{p}^{-1}=\left(  \frac{\lambda_{p,q(p)}}{\lambda_{p,p^{\ast}}}\right)
^{\frac{N}{p}}\left(  \frac{p+N(p-1)}{p}\right)  ^{N+1}I_{p}^{-1}%
\]
we obtain the second estimate in (\ref{linf1}) from (\ref{pp*}) and (\ref{l<h}).
\end{proof}

\begin{proof}
[Proof of Theorem \ref{main}]Of course, (\ref{linf2}) follows directly from
(\ref{linf1}).

Let us prove (\ref{l1=1}). If $0<q(p)<1,$ then H\"{o}lder's inequality yields
\[
1=\left\Vert u_{p,q(p)}\right\Vert _{q(p)}^{q(p)}\leq\left\Vert u_{p,q(p)}%
\right\Vert _{1}\left\vert \Omega\right\vert ^{1-q(p)}%
\]
so that%
\[
1=\lim_{p\rightarrow1^{+}}\frac{1}{\left\vert \Omega\right\vert ^{1-q(p)}}%
\leq\liminf_{p\rightarrow1^{+}}\left\Vert u_{p,q(p)}\right\Vert _{1}.
\]

As for $1\leq q(p)<p^{\ast},$ we first note from (\ref{linf1}) that%
\[
\lim_{p\rightarrow1^{+}}\left\Vert u_{p,q(p)}\right\Vert _{\infty}%
^{q(p)-1}=1.
\]
Then, taking into account that%
\[
1=\left\Vert u_{p,q(p)}\right\Vert _{q(p)}^{q(p)}\leq\left\Vert u_{p,q(p)}%
\right\Vert _{\infty}^{q(p)-1}\left\Vert u_{p,q(p)}\right\Vert _{1}%
\]
we get%
\[
1=\lim_{p\rightarrow1^{+}}\frac{1}{\left\Vert u_{p,q(p)}\right\Vert _{\infty
}^{q(p)-1}}\leq\liminf_{p\rightarrow1^{+}}\left\Vert u_{p,q(p)}\right\Vert
_{1}.
\]

We have thus proved the estimate
\[
1\leq\liminf_{p\rightarrow1^{+}}\left\Vert u_{p,q(p)}\right\Vert _{1}%
\]
which, in view of (\ref{l1<1}), leads us to (\ref{l1=1}).

Exploiting (\ref{Q}) with respect to $\lambda_{p,p}$ again (see (\ref{pb})),
we obtain from (\ref{pp}) and (\ref{l1=1}) that%
\[
h(\Omega)=\lim_{p\rightarrow1^{+}}\left(  \lambda_{p,p}^{-1}\left\vert
\Omega\right\vert ^{1-p}\left\Vert u_{p,q(p)}\right\Vert _{1}^{p}\right)
\leq\liminf_{p\rightarrow1^{+}}\lambda_{p,q(p)}.
\]
Bearing in mind (\ref{l<h}), this proves (\ref{l=h}).

In order to complete the proof, let us take $p_{n}\rightarrow1^{+}$ and set
\[
q_{n}:=q(p_{n})\text{ \ and \ }u_{n}:=u_{p_{n},q_{n}}.
\]

Then, $\lambda_{p_{n},q_{n}}=\left\Vert \nabla u_{n}\right\Vert _{p_{n}%
}^{p_{n}},$ $\left\Vert u_{n}\right\Vert _{q_{n}}=1,$ and $\lim_{n\rightarrow
\infty}q_{n}=1.$ Moreover, it follows from (\ref{l1=1}) that%
\begin{equation}
\lim_{n\rightarrow\infty}\left\Vert u_{n}\right\Vert _{1}=1. \label{a2}%
\end{equation}

We note that
\[
\left\vert Du_{n}\right\vert (\mathbb{R}^{N})=\left\vert Du_{n}\right\vert
(\Omega)=\left\Vert \nabla u_{n}\right\Vert _{1}\leq\left\Vert \nabla
u_{n}\right\Vert _{p_{n}}\left\vert \Omega\right\vert ^{1-\frac{1}{p_{n}}%
}=\lambda_{p_{n},q_{n}}^{\frac{1}{p_{n}}}\left\vert \Omega\right\vert
^{1-\frac{1}{p_{n}}}.
\]

Hence, (\ref{l=h}) implies that
\begin{equation}
\limsup_{n\rightarrow\infty}\left\vert Du_{n}\right\vert (\mathbb{R}^{N}%
)\leq\lim_{n\rightarrow\infty}\lambda_{p_{n},q_{n}}^{\frac{1}{p_{n}}%
}\left\vert \Omega\right\vert ^{1-\frac{1}{p_{n}}}=h(\Omega). \label{a1}%
\end{equation}

We conclude from (\ref{a2}) and (\ref{a1}) that the sequence $\left(
u_{n}\right)  $ is bounded in $BV(\Omega).$ Thus, by the compactness of the
embedding $BV(\Omega)\hookrightarrow L^{1}(\Omega)$ we can assume (up to
passing to a subsequence) that $u_{n}\rightarrow u,$ in $L^{1}(\Omega)$ and
also poinwise almost everywhere\thinspace in $\Omega.$ Extending $u_{n}$ as
zero on $\mathbb{R}^{N}\setminus\overline{\Omega}$ we have that $u_{n}$
converges in $L^{1}(\mathbb{R}^{N})$ to $u$ extended as zero on $\mathbb{R}%
^{N}\setminus\overline{\Omega}.$

Owing to (\ref{a2}) we have $\left\Vert u\right\Vert _{1}=1,$ which confirms
item (a) and also implies that $u\in BV_{0}(\Omega).$ Hence, it follows from
(\ref{hbv}) that
\[
h(\Omega)\leq\frac{\left\vert Du\right\vert (\mathbb{R}^{N})}{\left\Vert
u\right\Vert _{1}}=\left\vert Du\right\vert (\mathbb{R}^{N}).
\]

Moreover, Proposition \ref{bv} and (\ref{a1}) yield
\[
\left\vert Du\right\vert (\mathbb{R}^{N})\leq\liminf_{n\rightarrow\infty
}\left\vert Du_{n}\right\vert (\mathbb{R}^{N})\leq h(\Omega),
\]
showing that $\left\vert Du\right\vert (\mathbb{R}^{N})=h(\Omega).$ Then, item
(c) is consequence of Proposition \ref{bv1}.

Now, let us prove item (b). Let us fix $r>1$ and $\epsilon>0.$ As
$q_{n}\rightarrow1,$ we have that $q_{n}<r$ for all $n\geq n_{0}$ and some
$n_{0}\in\mathbb{N}.$ Moreover, owing to the second estimate in (\ref{linf1})
we can also assume that%
\begin{equation}
\left\Vert u_{n}\right\Vert _{\infty}\leq\frac{h(\Omega)^{N}}{\left\vert
\Omega\right\vert h(\Omega^{\star})^{N}}+\epsilon,\text{ \ for all }n\geq
n_{0}. \label{a3}%
\end{equation}

By H\"{o}lder's inequality, we have%
\[
1=\left\Vert u_{n}\right\Vert _{q_{n}}\leq\left\Vert u_{n}\right\Vert
_{r}\left\vert \Omega\right\vert ^{\frac{1}{q_{n}}-\frac{1}{r}},
\]
so that%
\begin{equation}
\left\vert \Omega\right\vert ^{\frac{1}{r}-\frac{1}{q_{n}}}\leq\left\Vert
u_{n}\right\Vert _{r},\text{ \ for all }n\geq n_{0}. \label{a4}%
\end{equation}

We also have%
\begin{equation}
\left\Vert u_{n}\right\Vert _{r}^{r}\leq\left\Vert u_{n}\right\Vert _{\infty
}^{r-1}\left\Vert u_{n}\right\Vert _{1}\leq\left(  \frac{h(\Omega)^{N}%
}{\left\vert \Omega\right\vert h(\Omega^{\star})^{N}}+\epsilon\right)
^{r-1}\left\Vert u_{n}\right\Vert _{1},\text{ \ for all }n\geq n_{0}.
\label{a5}%
\end{equation}

Convergence dominated theorem and (\ref{a3}) imply that $u_{n}\rightarrow u$
in $L^{r}(\Omega).$ Hence, (\ref{a4}) and (\ref{a5}) imply that%
\begin{equation}
\left\vert \Omega\right\vert ^{\frac{1}{r}-1}\leq\left\Vert u\right\Vert
_{r}\leq\left(  \frac{h(\Omega)^{N}}{\left\vert \Omega\right\vert
h(\Omega^{\star})^{N}}+\epsilon\right)  ^{\frac{r-1}{r}}\left\Vert
u\right\Vert _{1}^{\frac{1}{r}}=\left(  \frac{h(\Omega)^{N}}{\left\vert
\Omega\right\vert h(\Omega^{\star})^{N}}+\epsilon\right)  ^{\frac{r-1}{r}}.
\label{a6}%
\end{equation}

As $r$ and $\epsilon$ are arbitrarily fixed, (\ref{a6}) implies that $u\in
L^{\infty}(\Omega)$ and
\[
\left\vert \Omega\right\vert ^{-1}\leq\left\Vert u\right\Vert _{\infty}%
\leq\frac{h(\Omega)^{N}}{\left\vert \Omega\right\vert h(\Omega^{\star})^{N}}.
\]

\end{proof}

As mentioned in the Introduction, right after Corollary \ref{LE1}, Bueno and
Ercole proved in \cite{BE11} that%
\[
\lim_{p\rightarrow1^{+}}\left\Vert v_{p,1}\right\Vert _{1}^{1-p}%
=h(\Omega)=\left\Vert v_{p,1}\right\Vert _{\infty}^{1-p}.
\]
As $\lambda_{p,1}=\left\Vert v_{p,1}\right\Vert _{1}^{1-p},$ a fact that was
not noticed in \cite{BE11}, the first equality above leads directly to
\begin{equation}
\lim_{p\rightarrow1^{+}}\lambda_{p,1}=h(\Omega), \label{p1}%
\end{equation}
which is (\ref{l=h}) in the case where $q(p)\equiv1.$ Thus, (\ref{p1})
combined with (\ref{pp}) and the monotonicity of the function $q\mapsto
\lambda_{p,q}\left\vert \Omega\right\vert ^{\frac{p}{q}}$ also produces
(\ref{l=h}) for $q(p)\in(1,p).$ However, this combination does not lead to the
same result for $q(p)\in(0,1)\cup(p,p^{\ast})$ as, for example, $q(p)=p^{\beta
}$ with $\beta<0$ or $\beta>1$ (and $p$ close to $1^{+}$). Our approach
combining (\ref{pp}) with Lemma \ref{estim} provides a unified proof to
(\ref{l=h}) as well as allows us to estimate the limit function $u.$

\section{Proof of Theorem \ref{main2}\label{sec4}}

In this section we prove Theorem \ref{main2}, by applying Picone's inequality
to $w_{p,q(p)}$ and $e_{p},$ where $e_{p}$ is the first eigenfunction defined
in (\ref{ep}).

\begin{proof}
[Proof of Theorem \ref{main2}]As
\[
\lambda_{p,q(p)}\leq\frac{\left\Vert \nabla w_{p,q(p)}\right\Vert _{p}^{p}%
}{\left\Vert w_{p,q(p)}\right\Vert _{q(p)}^{p}}=\left\Vert w_{p,q(p)}%
\right\Vert _{q(p)}^{q(p)-p}\leq\left\Vert w_{p,q(p)}\right\Vert _{\infty
}^{q(p)-p}\left\vert \Omega\right\vert ^{\frac{q(p)-p}{q(p)}}%
\]
we obtain from (\ref{l=h}) that
\begin{equation}
h(\Omega)\leq\liminf_{p\rightarrow1^{+}}\left\Vert w_{p,q(p)}\right\Vert
_{q(p)}^{q(p)-p}\leq\liminf_{p\rightarrow1^{+}}\left\Vert w_{p,q(p)}%
\right\Vert _{\infty}^{q(p)-p}. \label{t2d}%
\end{equation}

Applying Picone's inequality and using that $w_{p,q(p)}$ is a weak solution to
(\ref{Dir1}) we find
\begin{align*}
\lambda_{p}\int_{\Omega}e_{p}^{p}\mathrm{d}x  &  =\int_{\Omega}\left\vert
\nabla e_{p}\right\vert ^{p}\mathrm{d}x\\
&  \geq\int_{\Omega}\left\vert \nabla w_{p,q(p)}\right\vert ^{p-2}\nabla
w_{p,q(p)}\cdot\nabla(\frac{e_{p}^{p}}{w_{p,q(p)}^{p-1}})\mathrm{d}x\\
&  =\int_{\Omega}w_{p,q(p)}^{q(p)-1}\frac{e_{p}^{p}}{w_{p,q(p)}^{p-1}%
}\mathrm{d}x=\int_{\Omega}w_{p,q(p)}^{q(p)-p}e_{p}^{p}\mathrm{d}x.
\end{align*}
Hence,
\begin{equation}
\left\Vert w_{p,q(p)}\right\Vert _{\infty}^{q(p)-p}\int_{\Omega}W_{p}%
^{q(p)-p}e_{p}^{p}\mathrm{d}x\leq\lambda_{p}\int_{\Omega}e_{p}^{p}\mathrm{d}x
\label{pic}%
\end{equation}
where
\[
W_{p}:=\frac{w_{p,q(p)}}{\left\Vert w_{p,q(p)}\right\Vert _{\infty}}.
\]

Now, let us assume that
\[
L:=\limsup_{p\rightarrow1^{+}}\left\Vert w_{p,q(p)}\right\Vert _{\infty
}^{q(p)-p}<\infty.
\]

Using again that $w_{p,q(p)}$ is a weak solution to (\ref{Dir1}), we have from
Lemma \ref{estim}, with $\lambda=\sigma=1,$ that%
\begin{equation}
C_{p}\left\Vert w_{p,q(p)}\right\Vert _{\infty}^{\frac{N(p-q(p))}{p}%
}\left\Vert w_{p,q(p)}\right\Vert _{\infty}\leq\left\Vert w_{p,q(p)}%
\right\Vert _{1} \label{t2a}%
\end{equation}
where%
\[
C_{p}:=\lambda_{p,p^{\ast}}^{\frac{N}{p}}\left(  \frac{p}{p+N(p-1)}\right)
^{N+1}\sigma%
{\displaystyle\int_{0}^{1}}
(1-\tau)^{\frac{N(p-1)}{p}}\mathrm{d}\tau.
\]

It follows from (\ref{t2a}) that%
\[
\frac{C_{p}}{\left\Vert w_{p,q(p)}\right\Vert _{\infty}^{\frac{N(q(p)-p)}{p}}%
}\leq\left\Vert W_{p}\right\Vert _{1},
\]
so that (\ref{pp*}) and (\ref{blamb5}) yield%
\begin{equation}
0<\frac{\left\vert \Omega\right\vert h(\Omega^{\star})^{N}}{L^{N}}\leq
\liminf_{p\rightarrow1^{+}}\left\Vert W_{p}\right\Vert _{1}. \label{t2c}%
\end{equation}

We also have that%
\[
\left\Vert \nabla w_{p,q(p)}\right\Vert _{1}\leq\left\vert \Omega\right\vert
^{1-\frac{1}{p}}\left\Vert \nabla w_{p,q(p)}\right\Vert _{p}=\left\vert
\Omega\right\vert ^{1-\frac{1}{p}}\left\Vert w_{p,q(p)}\right\Vert
_{q(p)}^{\frac{q(p)}{p}}%
\]
so that%
\[
\left\vert DW_{p}\right\vert (\Omega)=\left\Vert \nabla W_{p}\right\Vert
_{1}\leq\left\vert \Omega\right\vert ^{1-\frac{1}{p}}\left\Vert w_{p,q(p)}%
\right\Vert _{\infty}^{\frac{q(p)-p}{p}}\left\Vert W_{p}\right\Vert
_{q(p)}^{\frac{q(p)}{p}}.
\]
Hence, as $\left\Vert W_{p}\right\Vert _{q(p)}^{\frac{q(p)}{p}}\leq\left\vert
\Omega\right\vert ^{\frac{1}{p}}$ and $\left\Vert W_{p}\right\Vert _{1}%
\leq\left\vert \Omega\right\vert ,$ we conclude that the family $\left(
W_{p}\right)  $ is uniformly bounded in $BV(\Omega).$

Now, let $p_{n}\rightarrow1^{+}$ be such that
\[
\lim_{n\rightarrow\infty}\left\Vert w_{p_{n},q(p_{n})}\right\Vert _{\infty
}^{q(p_{n})-p_{n}}=L.
\]

Owing to the compactness of $BV(\Omega)\hookrightarrow L^{1}(\Omega)$ we can
assume (passing to subsequences, if necessary) that $W_{p_{n}}\rightarrow W$
in $L^{1}(\Omega)$ and also poinwise almost everywhere in $\Omega.$ It follows
from (\ref{t2c}) that $W>0$ a.e. in $\Omega$ and this implies that $W_{p_{n}%
}^{q(p_{n})-p_{n}}\rightarrow1$ poinwise almost everywhere in $\Omega.$ As
$\left\Vert W_{p_{n}}^{q(p_{n})-p_{n}}\right\Vert _{\infty}\leq1,$ dominated
convergence theorem and Lemma \ref{lem1} guarantee that%
\[
\lim_{n\rightarrow\infty}\int_{\Omega}W_{p_{n}}^{q(p_{n})-p_{n}}e_{p_{n}%
}^{p_{n}}\mathrm{d}x=\int_{\Omega}e_{p_{n}}^{p_{n}}\mathrm{d}x=\left\Vert
e\right\Vert _{1}>0.
\]
Hence, (\ref{pic}) and (\ref{pp}) yield%
\[
L\left\Vert e\right\Vert _{1}\leq h(\Omega)\left\Vert e\right\Vert _{1},
\]
so that $L\leq h(\Omega).$ Combining this inequality with (\ref{t2d}) we
conclude that
\[
h(\Omega)=\lim_{p\rightarrow1^{+}}\left\Vert w_{p,q(p)}\right\Vert
_{q(p)}^{q(p)-p}=\lim_{p\rightarrow1^{+}}\left\Vert w_{p,q(p)}\right\Vert
_{\infty}^{q(p)-p}.
\]

\end{proof}

\begin{remark}
\label{alter}One can derive from Remark \ref{s=q} that if $0<q(p)<p^{\ast},$
then
\[
\liminf_{p\rightarrow1^{+}}\left\Vert w_{p,q(p)}\right\Vert _{q(p)}%
^{q(p)-p}=\liminf_{p\rightarrow1^{+}}\left\Vert w_{p,q(p)}\right\Vert
_{\infty}^{q(p)-p}%
\]
and%
\[
\limsup_{p\rightarrow1^{+}}\left\Vert w_{p,q(p)}\right\Vert _{q(p)}%
^{q(p)-p}=\limsup_{p\rightarrow1^{+}}\left\Vert w_{p,q(p)}\right\Vert
_{\infty}^{q(p)-p}.
\]
Thus, the alternative (\ref{alt}) in the statement of Theorem \ref{main2} can
be replaced with
\[
\limsup_{p\rightarrow1^{+}}\left\Vert w_{p,q(p)}\right\Vert _{q(p)}%
^{q(p)-p}=+\infty.
\]

\end{remark}

\section*{Acknowledgments}

The author thanks the support of Fapemig/Brazil (RED-00133-21), FAPDF/Brazil
(04/2021) and CNPq/Brazil (305578/2020-0).

\end{document}